\documentclass[twoside,reqno]{amsart}

\usepackage{amsmath,amssymb,amsthm,amsfonts}

\usepackage{color}

\newtheorem{lemma}{Lemma}[section]
\newtheorem{theorem}{Theorem}[section]

\numberwithin{equation}{section}

\newcommand{\dis}{\displaystyle}

\newcommand{\R}{\mathbb{R}}

\newcommand{\semiG}{\mathbb{A}}

\newcommand{\comml}{[\![}
\newcommand{\commr}{]\!]}

\newcommand{\FM}{\mathbf{M}}
\newcommand{\FP}{\mathbf{P}}
\newcommand{\FL}{\mathbf{L}}
\newcommand{\FI}{\mathbf{I}}

\newcommand{\Fu}{\mathbf{u}}

\newcommand{\CD}{\mathcal{D}}
\newcommand{\CE}{\mathcal{E}}

\newcommand{\CT}{\mathcal{T}}

\newcommand{\na}{\nabla}

\newcommand{\al}{\alpha}
\newcommand{\be}{\beta}
\newcommand{\ga}{\gamma}

\newcommand{\la}{\lambda}
\newcommand{\de}{\delta}
\newcommand{\si}{\sigma}
\newcommand{\pa}{\partial}
\newcommand{\ka}{\kappa}
\newcommand{\eps}{\epsilon}

\newcommand{\vth}{\vartheta}

\newcommand{\De}{\Delta}
\newcommand{\Ga}{\Gamma}

\newcommand{\lag}{\langle}
\newcommand{\rag}{\rangle}

\newcommand{\trn}{|\!|\!|}

\begin{document}

\title[The Vlasov-Poisson-Landau System]{Global solutions to the Vlasov-Poisson-Landau System}

\author[R.-J. Duan]{Renjun Duan}
\address[RJD]{Department of Mathematics, The Chinese University of Hong Kong,
Shatin, Hong Kong}
\email{rjduan@math.cuhk.edu.hk}

\author[T. Yang]{Tong Yang}
\address[TY]{Department of Mathematics, City University of Hong  Kong,
Kowloon, Hong Kong
}
\email{matyang@cityu.edu.hk}

\author[H.-J. Zhao]{Huijiang Zhao}
\address[HJZ]{School of Mathematics and Statistics, Wuhan University, P.R.~China}
\email{hhjjzhao@hotmail.com}


\keywords{Vlasov-Poisson-Landau system, asymptotic stability, time-velocity weight, a priori estimates}

\subjclass[2010]{35Q20, 35Q83; 35B35, 35B40}



\begin{abstract}
Based on the recent  study on the Vlasov-Poisson-Boltzmann system with general angular cutoff potentials \cite{DYZ,DYZ-s}, we establish in this paper the global existence of  classical solutions to the Cauchy problem of the Vlasov-Poisson-Landau system that includes the Coulomb potential. This then provides a different approach on this topic from the
recent work  \cite{Guo-VPL}.
\end{abstract}

\maketitle
\thispagestyle{empty}

\setcounter{tocdepth}{1}
\tableofcontents

\section{Introduction}

In plasma physics, the binary grazing collision between particles (e.g., electrons and ions) can be  modeled by the Landau operator, cf.~\cite{TrKr}. There
are different approaches in establishing  the mathematical theories on
 the Landau equation, see \cite{AV,DL, Guo-L, Lio,SG,Vi} and references therein.
Recently, Guo \cite{Guo-VPL} made  progress in proving
the global existence of classical solutions
to  the Vlasov-Poisson-Landau (called VPL in the sequel for simplicity) system in a periodic box. Precisely, he successfully constructed global unique solutions to the Cauchy problem for initial data which have small weighted $H^2$ norms, but can have large $H^{N}$ $(N\geq 3)$ norms with high velocity moments, that
includes the most important case of Coulomb potential. This result is highly non-trivial on this important topic. Note that the same approach is used
in \cite{SZ} for the problem in the whole space.

Based on our recent study on the Vlasov-Poisson-Boltzmann system with general angular cutoff potentials \cite{DYZ,DYZ-s}, we establish in this short paper the global existence of classical solutions to the Cauchy problem on the VPL system in the whole space $\R^3$. Hence, it provides an alternative approach for
the study on  this topic compared to \cite{Guo-VPL}. We emphasize that the main motivation of the paper is to clarify how the approach that we developed in \cite{DYZ,DYZ-s} can be applied to  the VPL system so that we will
 not pursue here the optimal regularity and velocity moments on initial data.

Consider the following Cauchy problem for the VPL system which describes the dynamics of electrons with a constant ion background profile in the whole space,
\begin{equation}\label{cp}
    \left\{\begin{split}
     & \pa_t f + \xi \cdot \na_x f + \na_x\phi \cdot \na_\xi f =Q(f,f),\\
&\De_x \phi =\int_{\R^3} f\,d\xi-1,\quad \phi(x)\to 0\ \text{as}\ |x|\to \infty,\\
& f(0,x,\xi)=f_0(x,\xi).
    \end{split}\right.
\end{equation}
Here, $f=f(t,x,\xi)\geq 0$ represents the density distribution function of the particles (e.g., electrons) located at
$x=(x_1,x_2,x_3)\in\R^3$ with velocity
$\xi=(\xi_1,\xi_2,\xi_3)\in\R^3$ at time $t\geq 0$. The potential function $\phi=\phi(t,x)$ generating the
self-consistent electric field $\na_x\phi$  is coupled with
$f(t,x,\xi)$ through the Poisson equation, where the constant ion background profile is normalized to be unit. $Q$ is the Landau collision operator defined by
\begin{eqnarray*}
  Q(f,g)&=&\na_\xi\cdot \left\{\int_{\R^3}B(\xi-\xi')
    [f(\xi')\na_\xi g(\xi)-\na_\xi f(\xi')g(\xi)]\,d\xi'\right\}\\
    &=&\sum_{i,j=1}^3\pa_{i}\int_{\R^3} B^{ij}(\xi-\xi')[f(\xi')\pa_{j}g(\xi)-\pa_{j}f(\xi')g(\xi)]\,d\xi'
\end{eqnarray*}
with
\begin{equation*}
    B^{ij}(\xi)=\left(\de_{ij}-\frac{\xi_i\xi_j}{|\xi|^2}\right)|\xi|^{\ga+2},\quad -3\leq \ga<-2.
\end{equation*}
Here and in the sequel, we use $\pa_i=\pa_{\xi_i}$ for brevity. Note that
 $\ga=-3$ corresponds to the Coulomb potential.

Let $\FM=(2\pi )^{-3/2} e^{-|\xi|^2/2}$ be the normalized  Maxwellian.
By setting
$
f(t,x,\xi)-\FM = \FM^{1/2} u(t,x,\xi),
$
the Cauchy problem \eqref{cp} becomes
\begin{equation}\label{rcp}
    \left\{\begin{split}
    & \pa_t u + \xi \cdot \na_x u+ \na_x\phi \cdot \na_\xi u -\frac{1}{2} \xi \cdot \na_x\phi u-\na_x \phi \cdot \xi \FM^{1/2} +\FL u= \Ga(u,u),\\
& \De_x \phi=\int_{\R^3} \FM^{1/2} u \,d\xi, \quad \phi(x)\to 0\ \text{as}\ |x|\to \infty,\\
& u(0,x,\xi)=u_0(x,\xi)=\FM^{-1/2} (f_0-\FM),
    \end{split}\right.
\end{equation}
where
\begin{eqnarray*}
&\FL u =-{\FM^{-\frac 12}} \left\{Q\left(\FM, \FM^{1/2} u\right)+ Q\left(\FM^{1/2} u , \FM\right)\right\},\ \
\Ga(u,u)={\FM^{-\frac 12}} Q\left(\FM^{1/2}u,\FM^{1/2}u\right)
\end{eqnarray*}
are the linearized and nonlinear Landau collision  terms, respectively.

In order to state the global existence of solutions to \eqref{rcp}, we need
the following notations. The Landau collision frequency is given by
\begin{equation*}
    \si^{ij}(\xi)=B^{ij}\ast \FM (\xi)=\int_{\R^3} B^{ij}(\xi-\xi')\FM(\xi')\,d\xi'.
\end{equation*}
Similar to \cite{DYZ,DYZ-s}, we introduce the time-velocity weight
corresponding to the Landau operator:
\begin{equation}\label{weight}
    w_{\tau,\la}(t,\xi)=\lag \xi \rag^{{(\ga+2)\tau}}e^{\frac{\la \lag \xi \rag^2}{(1+t)^{\vth}}}, \quad \lag \xi \rag=\sqrt{1+|\xi|^2},\ \tau\in \R,\ 0\leq \la\ll 1,\ \vth>0.
\end{equation}
Note that $w_{\tau,\la}$ depends also on the parameter
$\vth$.  As in \cite{Guo-L} and \cite{SG}, we define the energy norm and the corresponding dissipation rate norm, respectively, by
\begin{equation*}
    |u(x)|_{\tau,\la}^2=\int_{\R^3}w_{\tau,\la}^2(t,\xi)|u|^2\,d\xi,\quad \|u\|_{\tau,\la}^2=\int_{\R^3}|u(x)|_{\tau,\la}^2\,dx,
\end{equation*}
and
\begin{equation*}
    |u(x)|_{\si,\tau,\la}^2=\sum_{i,j=1}^3\int_{\R^3}w_\tau^2(t,\xi)\left\{\si^{ij}
    \pa_iu\pa_j u+\si^{ij}\frac{\xi_i}{2}\frac{\xi_j}{2}|u|^2\right\}\,d\xi,
    \quad \|u\|_{\si,\tau,\la}^2=\int_{\R^3}|u(x)|_{\si,\tau,\la}^2\,dx.
\end{equation*}
Moreover, for an integer
 $N\geq 0$  and  a constant $\ell\geq N$, we define the energy norm
of a given $u=u(t,x,\xi)$ involving the space-velocity derivatives and the time-weighted energy norm, respectively, by
\begin{eqnarray}
&& \trn u(t) \trn_{N,\ell,\la}^2 =\sum_{|\al|+|\be|\leq N}\left\|\pa_\be^\al u(t)\right\|_{|\be|-\ell,\la}^2+\|\na_x \phi(t)\|_{H^N}^2,\label{def.tri}\\
&& X_{N,\ell,\la}(t)=\sup_{0\leq s\leq t}\trn u(s) \trn_{N,\ell,\la}^2+\sup_{0\leq s\leq t} (1+s)^{\frac{3}{2}}\trn u(s) \trn_{N,\ell-1,\la}^2\nonumber\\
&&\hspace{6cm}+\sup_{0\leq s\leq t} (1+s)^{2(1+\vth)}\|\na_x^2\phi (s)\|_{H^{N-1}}^2.\label{def.triX}
\end{eqnarray}
Here,  $\phi$ is determined by $u$ through
\begin{equation}\label{def.phi}
    \phi(t,x)=-\frac{1}{4\pi |x|}\ast_x\int_{\R^3} \FM^{1/2} u(t,x,\xi) \,d\xi.
\end{equation}
If $\la=0$ or $\tau=0$, then we drop the corresponding parameter in the subscript, for example $w_\tau=w_{\tau,0}$, $\|u\|_{\si}=\|u\|_{\si,0,0}$.

The result of this paper can now be stated as follows. More details will be explained at the end of this section.

\begin{theorem}\label{thm.m}
Let $-3\leq \ga<-2$, $N\geq 8$, $\ell_0>\frac{3}{2}$, $\ell\geq  1+\max\left\{N, \frac{\ell_0}{2}-\frac{1}{\ga+2}\right\}$, $0<\la \ll 1$, and $\vth=\frac{-\ga-4}{6\ga+4}\in\left[\frac{1}{14},\frac{1}{4}\right)$.
Assume that $f_0=\FM +\FM^{1/2}u_0\geq 0$ and
\begin{equation*}
  \iint_{\R^3\times \R^3} \FM^{1/2} u_0\,dxd\xi=0.
\end{equation*}
There exist constants $\eps_0>0$, $C_0>0$
such that if
\begin{equation}\label{thm.ge.1}
   Y_{N,\ell,\la}(0)\equiv \sum_{|\al|+|\be|\leq N}\left\|\pa_\be^\al u_0\right\|_{|\beta|-\ell,\la}+\left\|\left(1+|x|+|\xi|^{-\frac {(\ga+2)\ell_0}{2}}\right)u_0\right\|_{Z_1}\leq \eps_0,
\end{equation}
then the Cauchy problem \eqref{rcp} admits a unique global solution $u(t,x,\xi)$ satisfying $f(t,x,\xi)=\FM+\FM^{1/2}u(t,x,\xi)\geq 0$ and
\begin{equation}\label{thm.ge.2}
   \sup_{t\geq 0}X_{N,\ell,\la}(t)\leq C_0 Y_{N,\ell,\la}(0)^2.
\end{equation}
\end{theorem}

The proof of Theorem \ref{thm.m} is basically
along the same line as  \cite{DYZ-s} for the study on the Vlasov-Poisson-Boltzmann system with angular cutoff soft potentials. We  point out here the main differences from \cite{DYZ-s}. First of all, corresponding to the dissipation property of the linearized Landau operator stated in Lemma \ref{lem.FL} in the
next section, the exponent in the algebraic part of the weight function \eqref{weight} is chosen to be $(\ga+2)\tau$. In this way,
to  compensate one order of derivative in the velocity variable,  the extra velocity moment $\lag \xi\rag^{-(\ga+2)}$  grows slower than $\lag \xi\rag^{2}$ at large velocity when $-3\leq \ga<-2$. This can be used to control
the growth in the  velocity variable when dealing with the weighted estimate on the nonlinear term $\na_x\phi\cdot \na_\xi u$. The technique used here is different from the one in \cite{Guo-VPL}, where the velocity diffusion dissipation from the Landau operator was  used.

Another difference concerns the time-decay estimate on the potential force $\na_x\phi$. In the case of the Vlasov-Poisson-Boltzmann system, the nonlinear term contains at most one order  derivative on the perturbation
in the term  $\na_x\phi\cdot \na_\xi u$, and hence the time-decay estimate on $\|\na_x^2\phi\|^2_{H^{N-1}}\sim \|a\|^2_{H^{N-1}}$, particularly on $\|\na_x^{N-1}a\|^2$, can be obtained in term of the total energy functional $\trn u(t) \trn^2_{N,\ell-1}$. However, the nonlinear Landau operator contains second order of differentiation in the velocity variable. To overcome this, we will use the high order energy functional $\CE_{N,\ell,\la}^{\rm h}(t)$. To obtain the time-decay of $\CE_{N,\ell,\la}^{\rm h}(t)$,  the balance of $\lag \xi\rag^2/(1+t)^{1+\vth}$ and $\lag \xi\rag^{\ga+2}$ with $\ga<-2$ is used to get a time-decay coefficient in the dissipation term. That is,
\begin{equation*}
    \min_{\xi\in\R^3}\left\{\lag \xi\rag^{\ga+2},\frac{\lag \xi\rag^2}{(1+t)^{1+\vth}}\right\}=(1+t)^{\frac{(1+\vth)(\ga+2)}{-\ga}}
\end{equation*}
leads to
\begin{equation*}
    \CD_{N,\ell,q}(t)\geq \kappa (1+t)^{\frac{(1+\vth)(\ga+2)}{-\ga}} \CE_{N,\ell,q}^{\rm h}(t).
\end{equation*}
Note that $\CD_{N,\ell,q}(t)$ and $\CE_{N,\ell,\la}^{\rm h}(t)$ will be defined
in Sections 3 and 4 respectively.
In this way, as in  \cite{SG},   the time-decay on $\CE_{N,\ell,\la}^{\rm h}(t)$
  for a proper choice of $0<\vth\leq 1/4$ in terms of $\gamma$ will be given
in the Step 3 for the proof of Theorem \ref{thm.m} in Section \ref{sec4}.

The rest of this paper is arranged as follows. In the next two sections, we will state
 some lemmas related to the basic properties of $\FL$ and $\Ga(\cdot,\cdot)$, and also the weighted estimates on the nonlinear terms. In Section \ref{sec4}, we will give the proof of Theorem \ref{thm.m}.

\medskip
\noindent{\bf Notations.}~
Throughout this paper,  $C$  denotes
some generic positive (generally large) constant and $\ka$ denotes some generic
positive (generally small) constant, where both $C$ and
$\ka$ may take different values in different places.
$A\sim B$ means $\ka A\leq B \leq \frac{1}{\ka} A$. We use $L^2$ to denote the usual Hilbert spaces $L^2=L^2_{x,\xi}$, $L^2_x$ or $L^2_\xi$ with the norm $\|\cdot\|$, and use
$\lag \cdot,\cdot\rag$ to denote the inner product over $L^2_{x,\xi}$ or $L^2_\xi$.
For $q\geq 1$,  the  mixed  velocity-space Lebesgue
space $Z_q=L^2_\xi(L^q_x)=L^2(\R^3_\xi;L^q(\R^3_x))$ is used.
For
multi-indices $\al=(\al_1,\al_2,\al_3)$ and
$\be=(\be_1,\be_2,\be_3)$,  $\pa^{\al}_\be=\pa_x^\al\pa_\xi^\be=\pa_{x_1}^{\al_1}\pa_{x_2}^{\al_2}\pa_{x_3}^{\al_3}
    \pa_{\xi_1}^{\be_1}\pa_{\xi_2}^{\be_2}\pa_{\xi_3}^{\be_3}.
$
The length of $\al$ is $|\al|=\al_1+\al_2+\al_3$ and similar for $|\be|$.

\section{Preliminary}

In this section, we will state two lemmas about some basic properties
of the
Landau operator.
Given a vector-valued function $\Fu=(u_1,u_2,u_3)$, define
\begin{equation*}
   P_\xi \Fu=\frac{\xi \otimes \xi }{|\xi|^2}\Fu=\left\{\frac{\xi}{|\xi|}\cdot \Fu \right\}\frac{\xi}{|\xi|} ,\quad i.e.,\ ({P}_\xi \Fu)_i= \left\{\sum_{j=1}^3 \frac{\xi_j}{|\xi|}u_j\right\}\frac{\xi_i}{|\xi|}.
\end{equation*}
Concerning the equivalent characterization of the dissipation rate and the
dissipative property of the linearized Landau operator, the following lemma
was proved in \cite{Guo-L}.

\begin{lemma}[\cite{Guo-L}]\label{lem.FL}
By using the above notations, we have

\noindent
(i)
\begin{equation*}
    |u|_{\si,\tau,\la}^2 \sim \left|(1+|\xi|)^{\frac{\ga}{2}} {P}_\xi \na_\xi u\right|_{\tau,\la}^2+\left|(1+|\xi|)^{\frac{\ga+2}{2}}\{{I}-{P}_\xi\}\na_\xi u\right|_{\tau,\la}^2+\left|(1+|\xi|)^{\frac{\ga+2}{2}} u\right|_{\tau,\la}^2.
\end{equation*}

\noindent (ii)
$\lag \FL u,v\rag=\lag u,\FL v\rag$, $\lag \FL u,u\rag \geq 0$,
\begin{equation*}
    \ker \FL={\rm span}\left\{\FM^{1/2},\xi_i\FM^{1/2},1\leq i\leq 3,
    |\xi|^2\FM^{1/2} \right\},
\end{equation*}
and, $\lag \FL u,u\rag =0$ if and only if
$
    u=\FP u,
$
where
\begin{eqnarray}
&\dis \FP u= \left\{a(t,x)+b(t,x)\cdot \xi+c(t,x)\left(|\xi|^2-3\right)\right\}\FM^{1/2},\label{macro}\\
&\dis a= \int_{\R^3} \FM^{1/2} u\,d\xi,\
b_i=\int_{\R^3}\xi_i \FM^{1/2}u\,d\xi,\ \ 1\leq i\leq 3,\
c= \frac{1}{6}\int_{\R^3}\left(|\xi|^2-3\right) \FM^{1/2} u\,d\xi.\nonumber
\end{eqnarray}
Moreover, there exists $\kappa_0>0$ such that
\begin{equation*}
    \lag \FL u, u\rag \geq \kappa_0 \left|\{\FI-\FP\} u\right|_\si^2.
\end{equation*}

\end{lemma}

The following lemma states the weighted estimate on $\FL u$ and $\Ga(u,u)$. Notice that since $\la\geq 0$ is small enough, the coefficient $\la/(1+t)^{\vth}$ in front of $\lag \xi\rag^2$ in the exponential part of the weight function $w_{\tau,\la}$ is small  uniformly in  $t\geq 0$. Then, the proof of the
following  lemma follows directly from the argument used in \cite{SG}.

\begin{lemma}[\cite{SG}]

\noindent (i)
There exist $\kappa>0$ and  $C>0$,  such that
\begin{equation}
    \left\lag w_{\tau,\la}^2(t,\xi) \FL u, u\right\rag \geq \kappa |u|_{\si,\tau,\la}^2-C\left|\chi_{\{|\xi|\leq 2C\}} u\right|_\tau^2.\label{est.l}
\end{equation}
Let $|\be|>0$, $\tau=|\be|-\ell$ with $\ell\geq 0$. For $\eta>0$ small enough there exists $C_\eta>0$ such that
\begin{equation}\label{est.lv}
    \left\lag w_{\tau,\la}^2(t,\xi) \pa_\be\FL u, \pa_\be u\right\rag \geq \kappa |\pa_\be u|_{\si,\tau,\la}^2-\eta \sum_{|\be'|=|\be|} |\pa_{\be'} u|_{\si,\tau,\la}^2-C_\eta \sum_{|\be'|<|\be|} |\pa_{\be'} u|_{\si,|\be'|-\ell,\la}.
\end{equation}

\noindent (ii)
Let $N\geq 8$, $|\al|+|\be|\leq N$, $\tau=|\be|-\ell$ with $\ell\geq 0$. Then
\begin{eqnarray}
&\dis \left\lag w_{\tau,\la}^2(t,\xi) \pa_\be^\al \Ga(u_1,u_2),\pa_\be^\al u_3\right\rag\nonumber\\
&\dis \leq C \sum_{\substack{|\al'|+|\be'|\leq N  \\
\be''\leq \be'\leq \be }}\left\{\left|\pa_{\be''}^{\al'}u_1\right|_{\tau}\left|\pa_{\be-\be'}^{\al-\al'}u_2\right|_{\si,\tau,\la} +\left|\pa_{\be''}^{\al'} u_1\right|_{\si,\tau}\left|\pa_{\be-\be'}^{\al-\al'}u_2\right|_{\tau,\la}\right\}\left|\pa_\be^\al u_3\right|_{\si,\tau,\la}.\label{est.ga}
\end{eqnarray}

\end{lemma}

\section{Basic Lemmas}

To prove Theorem \ref{thm.m}, for brevity, we only focus on obtaining the uniform-in-time {\it a priori} estimate on the solution $u(t,x,\xi)$ to the Cauchy problem \eqref{rcp}. Recall \eqref{def.tri}. The main goal is to construct an  energy functional $\CE_{N,\ell,\la}(t)\sim \trn u(t) \trn_{N,\ell,\la}^2 $ such that the following estimate holds. Similar to \eqref{def.triX}, define
\begin{equation}\label{def.x}
    \widetilde{X}_{N,\ell,\la}(t)=\sup_{0\leq s\leq t}\CE_{N,\ell,\la}(s)+\sup_{0\leq s\leq t} (1+s)^{\frac{3}{2}}\CE_{N,\ell-1,\la}(s)
    +\sup_{0\leq s\leq t} (1+s)^{2(1+\vth)}\left\|\na_x^2\phi (s)\right\|_{H^{N-1}}^2.
\end{equation}
Then, by assuming that
$\widetilde{X}_{N,\ell,\la}(t)$ is small  in $0\leq t<T$ for given $T>0$, for the proper choice of parameters $N,\ell,\la,\vth$, we shall prove  that
\begin{equation}\label{est.X}
  \widetilde{X}_{N,\ell,\la}(t)\leq C\left\{Y_{N,\ell,\la}(0)^2+\widetilde{X}_{N,\ell,\la}(t)^2\right\}
\end{equation}
holds for $0\leq t<T$. Recall that $Y_{N,\ell,\la}(0)$ given in \eqref{thm.ge.1} depends only on initial data $u_0$. For later use, corresponding to the  energy functional $\CE_{N,\ell,\la}(t)$, we also define the functional for
the energy dissipation rate
\begin{multline}\label{def.ed}
   \CD_{N,\ell,\la}(t) =  \sum_{|\al|+|\be|\leq N}\left\|\pa_\be^\al\{\FI-\FP\} u(t)\right\|_{\si,|\be|-\ell,\la}^2+\|a\|^2+\sum_{|\al|\leq N-1}\left\|\na_x\pa^\al (a,b,c)\right\|^2\\
   +\frac{1}{(1+t)^{1+\vth}} \sum_{|\al|+|\be|\leq N}\left\|\pa_\be^\al\{\FI-\FP\} u(t)\right\|_{|\be|+\frac{1}{\ga+2}-\ell,\la}^2.
\end{multline}

The proof of \eqref{est.X} will be given in the next section. Here,  we prepare some more estimates on the nonlinear terms. The proofs follow from
the arguments used in \cite{DYZ-s} for the Vlasov-Poisson-Boltzmann system and hence most of details  will be omitted for brevity.
However,  when necessary, we will point out the key points in the proof and the main differences from \cite{DYZ-s}.

\begin{lemma}\label{lem.ga}
(i) Let $N\geq 4$, $\ell\geq 0$, $0<\la\ll 1$. Then
\begin{eqnarray}
&\dis \left|\left\lag \Ga(u,u),u\right\rag \right|\leq C\left\{\CE_{N,\ell,\la}(t)\right\}^{1/2} \left\{\left\|\na_x (a,b,c)\right\|_{H^1}^2+\|\{\FI-\FP\}u\|_\si^2\right\},\label{lem.ga.1}\\
&\dis \left|\left\lag \Ga(u,u),w_{-\ell,\la}^2(t,\xi)\{\FI-\FP\}u\right\rag\right|\leq  C\left\{\CE_{N,\ell,\la}(t)\right\}^{1/2} \left\{\left\|\na_x (a,b,c)\right\|_{H^1}^2+\|\{\FI-\FP\}u\|_{\si,\ell,\la}^2\right\}.\label{lem.ga.2}
\end{eqnarray}

(ii) Let $N\geq 8$, $1\leq |\al|+|\be|\leq N$, $\ell\geq |\be|$ and $0< \la \ll 1$. For $u=u(t,x,\xi)$, define $u_{\al\be}$ as $u_{\al\be}=\pa^\al u$ if $|\be|=0$ and $u_{\al\be}=\pa^\al_\be\{\FI-\FP\} u$ if $|\be|\geq 1$. Then
\begin{equation}\label{lem.ga.3}
\left|\left\lag \pa^\al_\be \Ga(u,u), w_{|\be|-\ell,\la}^2(t,\xi)u_{\al\be} \right\rag\right|
\leq  C\left\{\CE_{N,\ell,\la}(t)\right\}^{1/2}\left\{\left\|\na_x (a,b,c)\right\|_{H^{N-1}}^2+\CD_{N,\ell,\la}(t)\right\}.
\end{equation}
\end{lemma}

\begin{proof}
Apply the decomposition
\begin{equation}\label{ga.de}
\Ga(u,u)=\Ga(\FP u,\FP u)+\Ga(\FP u,\{\FI-\FP\}u)+\Ga(\{\FI-\FP\}u,\FP u)\\
+\Ga(\{\FI-\FP\}u,\{\FI-\FP\}u).
\end{equation}
All the  inner products on the left-hand side of \eqref{lem.ga.1}, \eqref{lem.ga.2} and \eqref{lem.ga.3} corresponding to the fourth term of \eqref{ga.de} can be estimated directly by using \eqref{est.ga}. And for those from the first three terms of \eqref{ga.de}, we recall the expression of $\Ga(u_1,u_2)$ from \cite{Guo-L,SG},
\begin{eqnarray*}
  \Ga(u_1,u_2) &=& \pa_i\left[\left\{B^{ij}\ast\left[\FM^{1/2}u_1\right]\right\}\pa_j u_2\right]-\left\{B^{ij}\ast\left[\frac{\xi_i}{2}\FM^{1/2} u_1\right]\right\}\pa_j u_2\\
  &&-\pa_i\left[\left\{B^{ij}\ast \left[\FM^{1/2}\pa_j u_1\right]\right\} u_2\right] +\left\{B^{ij} \ast\left[\frac{\xi_i}{2}\FM^{1/2}\pa_j u_1\right]\right\} u_2,
\end{eqnarray*}
and also recall \eqref{macro} for the expression of $\FP u$. Thus, similar to the proof of \eqref{est.ga} in \cite{SG}, one can follow  \cite{DYZ-s} to obtain the  estimates on these terms. This then completes the proof of Lemma \ref{lem.ga}.
\end{proof}

\begin{lemma}\label{lem.xiu}
(i) Let $u$ be the solution to \eqref{rcp}. Then
\begin{multline}\label{lem.xiu.1}
\left\lag \frac{1}{2} \xi\cdot \na_x\phi u, u\right\rag \leq \frac{1}{2} \frac{d}{dt}\int_{\R^3} |b|^2 (a+2c)\,dx \\
+C \left\{\|(a,b,c)\|_{H^2}+\|\na_x\phi\|_{H^1} + \|\na_x\phi\|\cdot \|\na_x b\|\right\}
\left\{\|\na_x (a,b,c)\|^2+\left\|{\lag \xi\rag^{\frac{\ga+2}{2}}}\{\FI-\FP\}u\right\|^2\right\}\\
+C\left \|\na_x^2\phi\right\|_{H^1} \left\|\lag \xi \rag^{1/2} \{\FI-\FP\} u\right\|^2.
\end{multline}
(ii) Let $N\geq 4$, $1\leq |\al|\leq N$, and $\ell\geq 0$. Then
\begin{multline}\label{lem.xiu.2}
\left\lag \pa^\al\left (\frac{1}{2} \xi\cdot \na_x\phi u\right), w_{-\ell,\la}^2(t,\xi) \pa^\al u \right\rag\\
\leq C\|\na_x^2\phi\|_{H^{N-1}}  \sum_{1\leq |\al|\leq N}\left\{\left\|\lag \xi \rag^{1/2} \pa^\al\{\FI-\FP\} u\right\|_{-\ell,\la}^2
+\left\|\pa^\al (a,b,c)\right\|^2\right\}.
\end{multline}
(iii) Let $N\geq 4$, $1\leq |\al|+|\be|\leq N$, $|\be|\geq 1$, and $\ell\geq |\be|$. Then
\begin{multline}\label{lem.xiu.3}
\left\lag \pa^\al_\be \left(\frac{1}{2} \xi\cdot \na_x\phi \{\FI-\FP\}u\right),w_{|\be|-\ell,\la}^2(t,\xi) \pa^\al_\be \{\FI-\FP\}u \right\rag \\
\leq C \|\na_x^2\phi\|_{H^{N-1}}\sum_{|\al|+|\be|\leq N } \left\|\lag \xi \rag^{1/2} \pa^\al_\be \{\FI-\FP\}u\right\|_{|\be|-\ell,\la}^2.
\end{multline}

\end{lemma}

\begin{proof}
To prove \eqref{lem.xiu.1}, notice that since $u$ is the solution to \eqref{rcp}, $u$ satisfies the momentum equation
\begin{equation*}
    \pa_t b +\na_x (a+2c)+\na_x \cdot \int_{\R^3} \xi \otimes \xi \{\FI-\FP\} u \,d\xi -\na_x \phi a= \na_x\phi.
\end{equation*}
Then \eqref{lem.xiu.1} follows from using the decomposition $u=\FP u+\{\FI-\FP\} u$ as in \eqref{ga.de} and the Cauchy-Schwarz inequality, where for low order term $\lag \frac{1}{2} \xi\cdot \na_x\phi \FP u, \FP u\rag$, as in \cite{DYZ-s}, we need to take the velocity integration  and then replace $\na_x\phi$ in terms of the equation above. It is then straightforward to obtain both \eqref{lem.xiu.2} and \eqref{lem.xiu.3} by using the Leibnitz rule, H\"{o}lder and Sobolev inequalities. The details proof of the lemma  are  omitted.
\end{proof}

\begin{lemma}\label{lem.pu}
(i) Let  $-3\leq \ga<-2$, $N\geq 4$, $1\leq |\al|\leq N$, and $\ell\geq 0$. Then
\begin{multline}\label{lem.pu.1}
    \left\lag \pa^\al \left(\na_x\phi \cdot \na_\xi u\right), w_{-\ell,\la}^2(t,\xi) \pa^\al u \right\rag\\
\leq C\left\|\na_x^2\phi\right\|_{H^{N-1}}\left\{\sum_{|\al|+|\be|\leq N,|\be|\leq 1}\left\|\lag \xi \rag \pa^\al_\be \{\FI-\FP\} u\|_{|\be|-\ell,\la}^2 +\|\na_x (a,b,c)\right\|_{H^{N-1}}^2\right\}.
\end{multline}
(ii) Let  $-3\leq \ga<-2$, $1\leq |\al|+|\be|\leq N$, $|\be|\geq 1$, and $\ell\geq |\be|$. Then
\begin{multline}\label{lem.pu.2}
\left\lag \pa^\al_\be \left(\na_x\phi \cdot \na_\xi \{\FI-\FP\}u\right),w_{|\be|-\ell,\la}^2(t,\xi) \pa^\al_\be \{\FI-\FP\}u \right\rag \\ \leq C\left\|\na_x^2\phi\right\|_{H^{N-1}} \sum_{|\al|+|\be|\leq N} \left\|\lag \xi \rag \pa_\be^\al \{\FI-\FP\}u\right\|_{|\be|-\ell,\la}^2.
\end{multline}
\end{lemma}

\begin{proof}
The proof is similar to the one for \eqref{lem.xiu.2} and \eqref{lem.xiu.3} through the Leibnitz rule, integrations by part in $\xi$, the H\"{o}lder and Sobolev inequalities. The difference lies in the fact that $\na_x\phi \cdot \na_\xi u$ and  $\na_x\phi \cdot \na_\xi \{\FI-\FP\}u$ involve the first order derivative
in the velocity variable so that the velocity weight has to be carefully distributed. In fact, to control those inner products on the left-hand side of \eqref{lem.pu.1} and \eqref{lem.pu.2},  for $|\be|\geq 0$ and $\ell\geq |\be|$, it suffices to have
\begin{equation*}
     w_{|\be|-\ell,\la}^2(t,\xi)\lesssim  \left\{\lag \xi \rag w_{1+|\be'|-\ell,\la}(t,\xi)\right\}\cdot \left\{\lag \xi \rag w_{|\be|-\ell,\la}(t,\xi)\right\},\quad |\be'|\leq |\be|,
\end{equation*}
which due to the definition of $w_{\tau,\la}$, is equivalent to require
\begin{equation*}
  2+ (\ga+2)+(\ga+2)(|\be'|-|\be|)\geq 0,\quad |\be'|\leq |\be|.
\end{equation*}
Since $-3\leq \ga<-2$, the above inequality is satisfied and this then
it completes the proof of Lemma \ref{lem.xiu}, cf. \cite{DYZ-s} for more details.
\end{proof}

\begin{lemma}\label{lem.G}
Let $G=\frac{1}{2} \xi\cdot \na_x\phi u-\na_x\phi \cdot \na_\xi u +\Ga(u,u)$. Let $-3\leq \ga<-2$, $N\geq 4$, $\ell_0>\frac{3}{2}$, and $\ell-1\geq \max\left\{N,\frac{\ell_0}{2}-\frac{1}{\ga+2}\right\}$. Then
\begin{equation}\label{lem.G.1}
  \left\|\lag \xi \rag^{-\frac{\ga+2}{2}\ell_0} G(t)\right\|_{Z_1}+\sum_{|\al|\leq 1}\left\|\lag \xi \rag^{-\frac{\ga+2}{2}\ell_0} \pa^\al G(t)\right\|\leq C \CE_{N,\ell-1}(t).
\end{equation}
\end{lemma}

\begin{proof}
Let $G_1=\frac{1}{2} \xi\cdot \na_x\phi u-\na_x\phi \cdot \na_\xi u$. Similar
to  \cite{DYZ-s}, as long as $\ell-1\geq \max\left\{N,\frac{\ell_0}{2}-\frac{1}{\ga+2}\right\}$ and $-3\leq \ga<-2$, we have
\begin{equation*}
   \left\|\lag \xi \rag^{-\frac{\ga+2}{2}\ell_0} G_1(t)\right\|_{Z_1} \leq C \|\na_x\phi\|\left\{\left\|w_{-\frac{1}{2}\ell_0}(t,\xi)\na_\xi u\right\|+\left\|w_{-\frac{1}{2}\ell_0+\frac{1}{\ga+2}}(t,\xi) u\right\|\right\}\leq C\CE_{N,\ell-1}(t)
\end{equation*}
and
\begin{equation*}
   \sum_{|\al|\leq 1}\left\|\lag \xi \rag^{-\frac{\ga+2}{2}\ell_0} \pa^\al G_1(t)\right\|\\
\leq C \|\na_x\phi\|_{H^3} \sum_{|\al|+|\be|\leq 2} \left\|w_{|\be|-(\ell-1)}(t,\xi)\pa_\be^\al u\right\| \leq C\CE_{N,\ell-1}(t).
\end{equation*}
Corresponding to $\Ga(u,u)$, for $|\al|\leq 1$,
\begin{eqnarray*}
&\dis \lag \xi \rag^{-\frac{\ga+2}{2}\ell_0}  \pa^\al \Ga(u,u)=w_{-\frac{\ell_0}{2}}(t,\xi)\pa^\al \Ga(u,u)=\sum_{|\al_1|\leq |\al|} C_{\al_1}^\al G_{\al_1}
\end{eqnarray*}
with
\begin{eqnarray*}
  G_{\al_1} &=& w_{-\frac{\ell_0}{2}}\left\{B^{ij}\ast\pa_i\left[\FM^{1/2}\pa^{\al_1}u\right]\right\}\pa_j \pa^{\al-\al_1}u
  +w_{-\frac{\ell_0}{2}}\left\{B^{ij}\ast\left[\FM^{1/2}\pa^{\al_1}u\right]\right\}\pa_i\pa_j \pa^{\al-\al_1}u\\
  &&-w_{-\frac{\ell_0}{2}}\left\{B^{ij}\ast\left[\frac{\xi_i}{2}\FM^{1/2} \pa^{\al_1}u\right]\right\}\pa_j \pa^{\al-\al_1}u\\
  &&-w_{-\frac{\ell_0}{2}}\left\{B^{ij}\ast \pa_i\left[\FM^{1/2}\pa_j \pa^{\al_1}u\right]\right\} \pa^{\al-\al_1}u
  -w_{-\frac{\ell_0}{2}}\left\{B^{ij}\ast \left[\FM^{1/2}\pa_j \pa^{\al_1}u\right]\right\} \pa_i\pa^{\al-\al_1}u \\
  &&+w_{-\frac{\ell_0}{2}}\left\{B^{ij} \ast \left[\frac{\xi_i}{2}\FM^{1/2}\pa_j \pa^{\al_1}u\right]\right\} \pa^{\al-\al_1}u,
\end{eqnarray*}
where the Einstein summation convention for repeated indices has been used.
We estimate the second term $G_{\al_1}^{(2)}$ of $G_{\al_1}$ first. By the Cauchy-Schwarz inequality,
\begin{eqnarray*}
&&\left\{B^{ij}\ast\left[\FM^{1/2}\pa^{\al_1}u\right]\right\}\leq \left[\left|B^{ij}\right|^2\ast \FM^{1/2}\right]^{1/2}\left\|\FM^{1/4}\pa^{\al_1}u\right\|_{L^2_\xi}\leq C \lag \xi \rag^{\ga+2}|\pa^{\al_1} u|_{-(\ell-1)},
\end{eqnarray*}
so that
\begin{eqnarray*}
&&\left|G_{\al_1}^{(2)}\right|\leq Cw_{-\frac{\ell_0}{2}+1}\left|\pa^{\al_1} u\right|_{-(\ell-1)}\left|\pa_i\pa_j \pa^{\al-\al_1}u\right|.
\end{eqnarray*}
This implies
\begin{multline*}
\left\|G_{\al_1}^{(2)}\right\|_{Z_1}+\left\|G_{\al_1}^{(2)}\right\|\leq C \left\|w_{-\frac{\ell_0}{2}+1}\|\pa^{\al_1} u\|_{-(\ell-1)}\|\pa_i\pa_j \pa^{\al-\al_1}u\|_{L^2_x}\right\|_{L^2_\xi}\\
+C \sup_{x\in \R^3}|\pa^{\al_1} u|_{-(\ell-1)}\times\left\|w_{-\frac{\ell_0}{2}+1}\pa_i\pa_j \pa^{\al-\al_1}u\right\|_{L^2_{x,\xi}}\leq C\CE_{N,\ell-1}(t),
\end{multline*}
because $-\frac{\ell_0}{2}+1\geq 2-(\ell-1)$ which is equivalent to $\ell-1\geq \frac{\ell_0}{2}+1$. It is direct to verify that for all other terms of $G_{\al_1}$, the same estimate still holds and hence it proves \eqref{lem.G.1} for the part $\Ga(u,u)$ in $G$. Thus, the proof of lemma  is completed.
\end{proof}

For later use, we also need the time-decay property of the linearized Landau system
\begin{equation*}
\pa_t u +\xi \cdot \na_x u -\na_x \phi \cdot \xi \FM^{1/2} =\FL u
\end{equation*}
with initial data given by $u_0(x,\xi)$, where $\phi$ is defined by \eqref{def.phi}. Denote the solution operator by $\semiG(t)$.

\begin{lemma}\label{thm.lide}
 Set $\mu=\mu(\xi):=\lag \xi \rag^{-\frac{{\ga+2}}{2}}$. Let $-3\leq \ga<-2$, $\ell\geq 0$, $\ell_0>3/2$, $\al\geq 0$, $m=|\al|$, and $\si_{m}=\frac{3}{4}+\frac{m}{2}$.  Assume
\begin{equation*}
    \int_{\R^3}a_0\,dx=0,\quad \int_{\R^3}(1+|x|)|a_0|\,dx<\infty,
\end{equation*}
and
\begin{equation*}
   \left\|\mu^{\ell+\ell_0}u_0\right\|_{Z_1}+\left\|\mu^{\ell+\ell_0} \pa^\al u_0\right\|<\infty.
\end{equation*}
Then, the evolution operator $\semiG(t)$ satisfies
\begin{multline*}
\left\|\mu^{\ell}\pa^\al \semiG(t)u_0\right\| +\left\|\pa^\al \na_x\De_x^{-1}\FP_0 \semiG(t) u_0\right\|\\
\leq C(1+t)^{-\si_{m}} \left(\left\|\mu^{\ell+\ell_0}u_0\right\|_{Z_1}+\left\|\mu^{\ell+\ell_0} \pa^\al u_0\right\|+\left\|(1+|x|)a_0\right\|_{L^1_x}\right)
\end{multline*}
for any $t\geq 0$, where $\FP_0$ is the projection given by $\FP_0 u=a(t,x)\FM^{1/2}$.
\end{lemma}

The proof of Lemma \ref{thm.lide} is completely the same as  the one in \cite{DYZ-s} and thus it is omitted for brevity. In fact, what we have changed in the Landau case is the definition of $\mu(\xi)$ due to the coercivity estimate
\begin{equation*}
    \lag \FL u, u\rag \gtrsim |\{\FI-\FP\} u|_\si^2
\end{equation*}
and the property
\begin{equation*}
   |\{\FI-\FP\} u|_\si^2 \gtrsim |\lag \xi \rag ^{\frac{\ga+2}{2}} \{\FI-\FP\} u|^2= |\mu^{-1} \{\FI-\FP\} u|^2.
\end{equation*}
In addition, to prove
Lemma \ref{thm.lide}, we need to use the weighted estimates on the linearized Landau operator $\FL$ in terms of \eqref{est.l}; see also \cite{SG} and \cite{St-Op}.

\section{Global existence}\label{sec4}

In the last section, we will give the proof of the main theorem in this paper.
\\

\noindent{\bf Proof of Theorem 1.1.} The proof is divided into
the following three steps.

\medskip
\noindent{\bf Step 1.}  Suppose that $u(t,x,\xi)$ is a smooth solution to the Cauchy problem \eqref{rcp} satisfying that
$\widetilde{X}_{N,\ell,\la}(t)$ is small  in $0\leq t<T$ for a given $T>0$, where parameters $N,\ell,\la,\vth$ are to be determined to ensure that this
smallness holds global in time. Let $-3\leq \ga<-2$, $N\geq 8$, $\ell\geq N$, $0<\la\ll 1$, and $\vth>0$. Then, we claim that there is $\CE_{N,\ell,\la}(t)$ such that for $0\leq t<T$,
\begin{equation}\label{egy.ineq}
    \frac{d}{dt} \CE_{N,\ell,\la}(t)+\kappa \CD_{N,\ell,\la}(t)\leq 0,
\end{equation}
and also  there is a high order energy functional $\CE_{N,\ell,\la}^{\rm h}(t)$ satisfying
\begin{equation}\label{def.eh}
  \CE_{N,\ell,\la}^{\rm h}(t) \sim\sum_{|\al|+|\be|\leq N}\left\|\pa_\be^\al \{\FI-\FP\}u(t)\right\|_{|\be|-\ell,\la}^2+\left\|\na_x^2\phi\right\|_{H^{N-1}}^2
+\|a\|^2+ \sum_{|\al|\leq N-1}\left\|\na_x\pa^\al (a,b,c)\right\|^2,
\end{equation}
such that for $0\leq t<T$,
\begin{equation}\label{lem.ee.h.1}
    \frac{d}{dt} \CE_{N,\ell,\la}^{\rm h}(t)+\kappa \CD_{N,\ell,\la}(t)\leq C\|\na_x (a,b,c)\|^2.
\end{equation}
Note that this implies that
\begin{equation}\label{lem.X.pp4}
  \sup_{0\leq s\leq t}\CE_{N,\ell,\la}(s)\leq \CE_{N,\ell,\la}(0).
\end{equation}

\begin{proof}[Proof of \eqref{egy.ineq}]
First notice that the a priori assumption implies
\begin{equation}\label{phi.de}
    \|\na_x^2\phi\|_{H^{N-1}}\leq \frac{C \de}{(1+t)^\vth},
\end{equation}
where $\de>0$ is sufficiently small. Then, with the help of \eqref{est.l}, \eqref{est.lv}, and Lemmas \ref{lem.ga}, \ref{lem.xiu} and \ref{lem.pu}, following the argument in \cite{DYZ-s}, one can obtain  \eqref{egy.ineq} by standard energy
method. Here, we only point out the key points in the proof. As in \cite{Guo-L}, regarding the estimates on the linear transport terms, we can use
\begin{equation*}
w_{|\be|-\ell, \la}^2(t,\xi)=w_{|\be|+\frac{1}{2}-\ell, \la}(t,\xi)w_{|\be_2|+\frac{1}{2}-\ell, \la}(t,\xi)
\leq C \lag \xi\rag^{\frac{\ga+2}{2}} w_{|\be|-\ell, \la}(t,\xi)\cdot \lag \xi\rag^{\frac{\ga+2}{2}} w_{|\be_2|-\ell, \la}(t,\xi),
\end{equation*}
where $\be=\be_1+\be_2$, $|\be_1|=1$. Moreover, due to the time-dependent exponential factor $\exp\{\la \lag \xi \rag^2/(1+t)^\vth\}$ in the velocity weight function, the weighted energy estimates indeed generate the last part of the energy dissipation rate given in \eqref{def.ed}. And this part can be used to absorb those nonlinear terms with possible velocity growth that appear on the right-hand of \eqref{lem.xiu.1}-\eqref{lem.xiu.3} and \eqref{lem.pu.1}-\eqref{lem.pu.2} by using \eqref{phi.de}. In addition, to obtain the dissipation of the macroscopic component $(a,b,c)$ in $\CD_{N,\ell,\la}(t)$, one can apply  the following fluid-type moment system, cf. \cite{Guo-L},
\begin{equation*}
    \left\{\begin{array}{l}
       \dis \pa_t a + \na_x\cdot b =0,\\
\dis \pa_t b+ \na_x (a + 2c) + \na_x \cdot \Theta (\{\FI-\FP\} u) -\na_x\phi=\na_x\phi a,\\
\dis \pa_t c + \frac{1}{3}\na_x\cdot b +\frac{5}{3} \na_x\cdot  \Lambda (\{\FI-\FP\} u)=\frac{1}{3} \na_x\phi \cdot b,\\
\dis \De_x \phi =a,
    \end{array}\right.
\end{equation*}
and
\begin{equation*}
    \left\{\begin{array}{l}
\dis \pa_t \Theta_{ij}(\{\FI-\FP\} u)+\pa_{x_i} b_j +\pa_{x_j} b_i -\frac{2}{3} \de_{ij} \na_x\cdot b -\frac{10}{3} \de_{ij}\na_x\cdot \Lambda (\{\FI-\FP\} u)\\
\dis \qquad\qquad =\Theta_{ij} (r+G)-\frac{2}{3}\de_{ij}\na_x\phi \cdot b,\\
\dis \pa_t \Lambda_i (\{\FI-\FP\} u) +\pa_i c =\Lambda_i(r+G),
    \end{array}\right.
\end{equation*}
with
$
    r=-\xi\cdot \na_x \{\FI-\FP\}u+\FL u$ and $G$ given in Lemma \ref{lem.G}, where as in \cite{DS-VPB}, we use the notations
\begin{equation*}
    \Theta_{ij}(u)=\int_{\R^3}(\xi_i\xi_j -1)\FM^{1/2} u\,d\xi,\ \ \Lambda_i(u)=\frac{1}{10}\int_{\R^3} (|\xi|^2-5)\xi_i\FM^{1/2} u\,d\xi.
\end{equation*}
Finally, as in \cite{DYZ-s},  $\CE_{N,\ell,\la}(t)$ can be defined by
\begin{multline*}
 \CE_{N,\ell,\la}(t)
 =M_3\left[M_2\left\{\frac{M_1}{2}\left[\sum\limits_{|\al|\leq N}\left(\left\|\pa^\al u\right\|^2+\left\|\pa^\al \na_x\phi\right\|^2\right)-{\displaystyle\int_{\R^3}} |b|^2(a+2c)\,dx\right]+\CE_N^{\rm int}(t)\right\}\right.\\
\left.+\left\|\{\FI-\FP\} u\right\|_{-\ell,\la}^2+\sum\limits_{1\leq |\al|\leq N}\left\|\pa^\al u\right\|_{-\ell,\la}^2\right]+\sum\limits_{m=1}^NC_m\sum\limits_{
|\al|+|\be|\leq N, |\be|= m}\left\|\pa_\be^\al \{\FI-\FP\} u\right\|_{|\be|-\ell,\la}^2
\end{multline*}
together with
\begin{multline*}
  \CE^{{\rm int}}_{N}(t) =\sum_{|\al|\leq N-1}\int_{\R^3}\na_x\pa^\al c\cdot \Lambda(\pa^\al\{\FI-\FP\}u)\,dx -\kappa\sum_{|\al|\leq N-1}\int_{\R^3}\pa^\al a \pa^\al\na_x\cdot b \,dx\\
  +\sum_{|\al|\leq N-1}\sum_{i,j=1}^3\int_{\R^3}
\left(\pa_{x_i}\pa^\al b_j+\pa_{x_j}\pa^\al b_i-\frac{2}{3}\de_{ij}\na_x\cdot \pa^\al b\right)
\Theta_{ij}(\pa^\al\{\FI-\FP\}u)\,dx,
\end{multline*}
where $0<\kappa<1$ is a small constant, $C_m>0$ $(1\leq m\leq N)$ are properly chosen constants, and $M_i$ $(i=1,2,3)$ are constants large enough.
\end{proof}

\begin{proof}[Proof of \eqref{lem.ee.h.1}]
Compared with \eqref{egy.ineq}, the main difference in the proof of \eqref{lem.ee.h.1} comes from the estimation on $\|\{\FI-\FP\}u\|$ and $\|\{\FI-\FP\}u\|_{-\ell,\la}$. Notice that
the time evolution of $\{\FI-\FP\}u$ satisfies
\begin{multline*}
\pa_t \{\FI-\FP\}u+\xi \cdot \na_x \{\FI-\FP\}u + \na_x\phi \cdot \na_\xi \{\FI-\FP\}u
+\FL\{\FI-\FP\}u\\
=\Ga(u,u)+\frac{1}{2} \xi \cdot \na_x \phi \{\FI-\FP\}u +\comml \FP, \CT_{\phi}\commr u,
\end{multline*}
where $\comml A, B\commr=AB-BA$ denotes the commutator of two operators $A,B$, and $\CT_\phi$ is given by
\begin{equation*}
    \CT_\phi=\xi \cdot \na_x  +\na_x\phi \cdot \na_\xi  -\frac{1}{2} \xi \cdot \na_x\phi.
\end{equation*}
Therefore, \eqref{lem.ee.h.1} follows from the similar weighted energy estimates  as in \cite{DS-VMB}.
\end{proof}

\medskip
\noindent{\bf Step 2.} By Lemma \ref{lem.G}, it is straightforward to verify
\begin{equation}\label{lem.made.1}
  \|(a,b,c)(t)\|+\|\na_x\phi(t)\|\leq C(1+t)^{-\frac{3}{4}}\left \{Y_{N,\ell}(0)+\widetilde{X}_{N,\ell}(t)\right\}
\end{equation}
and
\begin{equation}\label{lem.made.2}
  \|\na_x(a,b,c)(t)\|\leq C(1+t)^{-\frac{5}{4}} \left\{Y_{N,\ell}(0)+\widetilde{X}_{N,\ell}(t)\right\}
\end{equation}
for $0\leq t<T$. Then, as in \cite{DYZ-s}, by the time-weighted estimate, we have from \eqref{egy.ineq} and \eqref{lem.made.1} that
\begin{equation}\label{ap.td1}
    \sup_{0\leq s\leq t} (1+s)^{\frac{3}{2}} \CE_{N,\ell-1,\la}(s)\leq C \left\{Y_{N,\ell,\la}(0)^2+\widetilde{X}_{N,\ell}(t)^2\right\}.
\end{equation}
Here, notice that we have used
\begin{equation*}
    \lag \xi \rag^{\frac{\ga+2}{2}} w_{|\be|-\ell,\la}(t,\xi)\sim w_{|\be|+\frac{1}{2}-\ell,\la}(t,\xi),
\end{equation*}
so that
\begin{equation*}
  \CD_{N,\ell,\la}(t)+ \|(b,c)(t)\|^2+\|\na_x\phi(t)\|^2\gtrsim \CE_{N,\ell-\frac{1}{2},\la}(t).
\end{equation*}

\medskip
\noindent{\bf Step 3.} Observe that
\begin{equation*}
    \min_{\xi\in\R^3}\left\{\lag \xi\rag^{\ga+2},\frac{\lag \xi\rag^2}{(1+t)^{1+\vth}}\right\}=(1+t)^{\frac{(1+\vth)(\ga+2)}{-\ga}},
\end{equation*}
where the equality is taken when $\lag \xi\rag^{\ga+2}=\lag \xi\rag^2/(1+t)^{1+\vth}$, i.e. $\lag \xi\rag=(1+t)^{\frac{1+\vth}{-\ga}}$. Set $p=1+\frac{(1+\vth)(\ga+2)}{-\ga}$. Then, by the choice  of $\vth$ given in Theorem \ref{thm.m}, since $-3\leq \ga<-2$,
one has $1/14\leq \vth<1/4$ and $p=\frac{1}{2}+2\vth$ with $0<p<1$. Therefore, recalling \eqref{def.eh} for the equivalent property of $\CE_{N,\ell,q}^{\rm h}(t)$, it follows from \eqref{lem.ee.h.1}  that
\begin{equation*}
    \frac{d}{dt}\CE_{N,\ell,\la}^{\rm h}(t)+\kappa p (1+t)^{p-1} \CE_{N,\ell,\la}^{\rm h}(t)\leq C\|\na_x (a,b,c)\|^2,
\end{equation*}
which after multiplying it by $e^{\kappa (1+t)^p}$ and taking the time integration over $[0,t]$, implies
\begin{equation*}
   \CE_{N,\ell,\la}^{\rm h}(t)\leq \CE_{N,\ell,\la}^{\rm h}(0)e^{-\kappa (1+t)^p}
+Ce^{-\kappa (1+t)^p}\int_0^t \|\na_x (a,b,c)(s)\|^2 e^{\kappa (1+s)^p}\,ds.
\end{equation*}
Using \eqref{lem.made.2} and
\begin{equation*}
   \int_0^t (1+s)^{-\frac{5}{2}} e^{\kappa (1+s)^p}\,ds\leq C (1+t)^{-(\frac{3}{2}+p)}  e^{\kappa (1+t)^p},
\end{equation*}
one has
\begin{equation*}
   \CE_{N,\ell,\la}^{\rm h}(t)\leq C(1+t)^{-2(1+\vth)} \left\{\CE_{N,\ell,\la}^{\rm h}(0)+Y_{N,\ell,\la}(0)^2+\widetilde{X}_{N,\ell}(t)^2\right\}.
\end{equation*}
Noticing $\left\|\na_x^2\phi \right\|_{H^{N-1}}^2\leq C  \CE_{N,\ell,\la}^{\rm h}(t) $, the above inequality implies
\begin{equation}\label{ap.tdh}
\sup_{0\leq s\leq t} (1+s)^{2(1+\vth)}\|\na_x^2\phi (s)\|_{H^{N-1}}^2\\
\leq C\sup_{0\leq s\leq t} (1+s)^{2(1+\vth)} \CE_{N,\ell,\la}^{\rm h}(t)\leq C \left\{Y_{N,\ell,\la}(0)^2+\widetilde{X}_{N,\ell}(t)^2\right\}.
\end{equation}
Recall \eqref{def.x} and notice $\widetilde{X}_{N,\ell}(t)\leq \widetilde{X}_{N,\ell,\la}(t)$. \eqref{ap.tdh} together with \eqref{lem.X.pp4} and \eqref{ap.td1} then prove \eqref{est.X} which is equivalent to \eqref{thm.ge.2}. Hence,
it completes the proof of Theorem \ref{thm.m}. \qed

\medskip
\noindent
{\bf Acknowledgements:}
The research of the first author was supported by the General Research Fund (Project No.: CUHK 400511) from RGC of Hong Kong. The research of the second author was supported by the General Research Fund of Hong Kong, CityU No. 104511, and the Croucher Foundation. And research of the third author was supported by the grants
from the National Natural Science Foundation of China under
contracts 10871151 and 10925103. This work is also
supported by ``the Fundamental Research Funds for the Central
Universities".

\end{document}